\documentclass{amsart}
\usepackage{amssymb}

\newtheorem{theorem}{Theorem}[section]
\newtheorem{proposition}[theorem]{Proposition}
\newtheorem{lemma}[theorem]{Lemma}

\theoremstyle{definition}
\newtheorem{definition}[theorem]{Definition}

\newtheorem{remark}[theorem]{Remark}

\numberwithin{equation}{section}

\begin{document}
\title[Special values and transformations]
{Zeros of $p$-adic hypergeometric functions, $p$-adic analogues of Kummer's and Pfaff's identities}


{}
\author{Neelam Saikia}
\address{Department of Mathematics, Indian Institute of Technology Guwahati, North Guwahati, Guwahati-781039, Assam, INDIA}
\curraddr{}
\email{neelam16@iitg.ac.in\\
nlmsaikia1@gmail.com}
\thanks{}


\thanks{The author acknowledges
the financial support of Department of Science and Technology, Government of India for financial support under INSPIRE Faculty Award.}
\begin{abstract} 
We classify all the zeros and non-zero values of a family of hypergeometric series in the $p$-adic setting. These values of hypergeometric series in the $p$-adic setting lead to transformations of hypergeometric series in the $p$-adic setting which can be described as $p$-adic analogues of Kummer's and Pfaff's linear transformations on classical hypergeometric series. We also evaluate certain summation identities for hypergeometric series in the $p$-adic setting as well as Gaussian hypergeometric series.
\end{abstract}
\maketitle
\section{Introduction and statement of results}
The main goal of this paper is to study zeros of hypergeometric series in the $p$-adic setting introduced by D. McCarthy \cite{mccarthy-ijnt-2, mccarthy-pacific}. We also establish analogues of classical hypergeometric series transformations, particularly very special cases of Kummer's and Pfaff's linear transformations, for hypergeometric series in the $p$-adic setting. This type of questions were posed by D. McCarthy \cite{mccarthy-pacific}. We now begin with the definition of classical hypergeometric series.
For a complex number $a$ and a non negative integer $k$ the rising factorial denoted by $(a)_k$ is defined by $(a)_k:=a(a+1)(a+2)\cdots(a+k-1)$ for $k>0$ and $(a)_0:=1.$ Then for $a_i,b_i,\lambda\in\mathbb{C}$ with $b_i\not\in\{\ldots,-3,-2,-1,0\},$ the classical hypergeometric series ${_{r+1}F_r}$ is defined by
\begin{align*}
{_{r+1}}F_{r}\left(\begin{array}{cccc}
                   a_1, & a_2, & \ldots, & a_{r+1} \\
                    & b_1, & \ldots, & b_r
                 \end{array}\mid \lambda
\right):=\sum_{k=0}^{\infty}\frac{(a_1)_k\cdots(a_{r+1})_k}
{(b_1)_k\cdots(b_r)_k}\cdot\frac{\lambda^k}{k!}.
\end{align*}
This series converges for $|\lambda|<1.$ Classical hypergeomeometric series play important role in different areas of mathematics. For example, they have significant applications in modular forms, elliptic curves, representation theory, differential equations etc. \cite{mccarthy-proc, mccarthy-ijnt, mortenson}. J. Greene \cite{greene} introduced the notion of hypergeometric series over finite fields which are finite field analogues of classical hypergeometric series.
Let $p$ be an odd prime and $\mathbb{F}_p$ denote a finite field with $p$ elements. Let $\widehat{\mathbb{F}_p^\times}$ denote the group of all multiplicative characters of $\mathbb{F}_p^{\times}$ and $\overline{\chi}$ denote the inverse of a multiplicative character $\chi$. We extend the domain of each $\chi\in\widehat{\mathbb{F}_p^\times}$ to $\mathbb{F}_p$ by simply setting $\chi(0):=0$ including the trivial character $\varepsilon.$ 
For multiplicative characters $\chi$ and $\psi$ of $\mathbb{F}_p$ the Jacobi sum is defined by
\begin{align}\label{jacobi}
J(\chi,\psi):=\sum_{y\in\mathbb{F}_p}\chi(y)\psi(1-y),
\end{align}
and the normalized Jacobi sum known as binomial is defined by
\begin{align}\label{binomial}
{\chi\choose \psi}:=\frac{\psi(-1)}{p}J(\chi,\overline{\psi}).
\end{align}
Let $n$ be a non negative integer. For multiplicative characters $A_1,A_2,\ldots,A_{n+1},$ and $B_1,B_2,\ldots,B_n$ of 
$\mathbb{F}_p$ with $t\in\mathbb{F}_p,$ J. Greene \cite{greene} defined ${_{n+1}F_n}(\cdots)$ hypergeometric function over finite field 
$\mathbb{F}_p$ by
\begin{align*}
{_{n+1}}F_{n}\left(\begin{array}{cccc}
                   A_1, & A_2, & \ldots, & A_{n+1} \\
                    & B_1, & \ldots, & B_n
                 \end{array}\mid t
\right):=\frac{p}{p-1}\sum_{\chi\in\widehat{\mathbb{F}_p^\times}}
{A_1\chi\choose\chi}\cdots{A_{n+1}\chi\choose B_n\chi}\chi(t).
\end{align*}
This function is also known as Gaussian hypergeometric function.
These functions were developed to have a parallel study of character sums analogous to special functions. 
Gaussian hypergeometric functions satisfy many identities which are often analogues of classical hypergeometric series identities, for more details, see \cite{greene}. Since the entries of the Gaussian hypergeometric function are multiplicative characters so results involving Gaussian hypergeometric functions often be restricted to primes in certain congruence classes for the existence of characters of specific orders, see for example \cite{EG, fuselier2, lennon-1, lennon-2}. To overcome these limitations, D. McCarthy \cite{mccarthy-ijnt-2, mccarthy-pacific} defined a function $_nG_n[\cdots]$ in terms of quotients of $p$-adic gamma functions which can be best described as an analogue of classical hypergeometric series in the $p$-adic setting. Let $\mathbb{Z}_p$ and $\mathbb{Q}_p$ denote the ring of $p$-adic integers and the field of $p$-adic numbers, respectively.
Let $\Gamma_p(\cdot)$ denote the Morita's $p$-adic gamma function. Let $\omega$ denote the Teichm\"{u}ller character of $\mathbb{F}_p,$ satisfying $\omega(a)\equiv a\pmod{p},$ and $\overline{\omega}$ denote the character inverse of $\omega$. For $x\in\mathbb{Q}$ let $\lfloor x\rfloor$ denote the greatest integer less than or equal to $x$ and $\langle x\rangle$ denote the fractional part of $x$, satisfying $0\leq\langle x\rangle<1$.
We now recall the McCarthy's hypergeometric function $_{n}G_{n}[\cdots]$ in the $p$-adic setting.
\begin{definition}\cite[Definition 5.1]{mccarthy-pacific} \label{defin1}
Let $p$ be an odd prime and $t \in \mathbb{F}_p$.
For positive integer $n$ and $1\leq k\leq n$, let $a_k$, $b_k$ $\in \mathbb{Q}\cap \mathbb{Z}_p$.
Then the function $_{n}G_{n}[\cdots]$ is defined as
\begin{align}
&{_nG_n}\left[\begin{array}{cccc}
             a_1, & a_2, & \ldots, & a_n \\
             b_1, & b_2, & \ldots, & b_n
           \end{array}|t
 \right]:=\frac{-1}{p-1}\sum_{a=0}^{p-2}(-1)^{an}~~\overline{\omega}^a(t)\notag\\
&\times \prod\limits_{k=1}^n(-p)^{-\lfloor \langle a_k \rangle-\frac{a}{p-1} \rfloor -\lfloor\langle -b_k \rangle +\frac{a}{p-1}\rfloor}
 \frac{\Gamma_p(\langle a_k-\frac{a}{p-1}\rangle)}{\Gamma_p(\langle a_k \rangle)}
 \frac{\Gamma_p(\langle -b_k+\frac{a}{p-1} \rangle)}{\Gamma_p(\langle -b_k \rangle)}.\notag
\end{align}
\end{definition}
This function is also known as $p$-adic hypergeometric function. It is clear from the definition \ref{defin1} that the value of $_nG_n[\cdots]$ function depends only on the fractional part of the parameters $a_k$ and $b_k$. Therefore, we may assume that $0\leq a_k,b_k<1.$ Gaussian hypergeometric functions satisfy many powerful transformation formulas that are often mirror symmetrical to their classical counterparts, for details see \cite{greene}. Note that these results can be converted into identities involving $_{n}G_{n}[\cdots]$ via the transformations 
\cite[Lemma 3.3]{mccarthy-pacific} and \cite[Proposition 2.5]{mccarthy-ffa} between finite field hypergeometric function and $p$-adic hypergeometric series. However, the new identities involving $_{n}G_{n}[\cdots]$ will be valid for the primes $p$, where the original characters existed over $\mathbb{F}_p.$
Therefore, it will be interesting to extend such results to almost all primes. In \cite{fm}, Fuselier-McCarthy established certain transformation identities for $p$-adic hypergeometric series in full generality. In particular, they proved a transformation result analogous to a Whipple's result for ${_3F_2}$-classical hypergeometric series. These transformations eventually leaded to settle one supercongruence conjecture of Rodriguez-Villegas between a truncated 
${_4F_3}$-classical hypergeometric series and the Fourier coefficients of a certain weight four modular form. This is one of the motivation to study transformation formulas with the expectation that transformation formulas will lead to new identities.
 Let  $\chi_4,$ be a multiplicative character of $\mathbb{F}_p$
of orders 4. Also, let $\varphi$ be the quadratic character of $\mathbb{F}_p.$ Consider the classical hypergeometric series
${_2F_1}\left(\begin{array}{cc}
            \frac{1}{4} , & \frac{3}{4}\vspace{1mm} \\
              & \frac{1}{2}
           \end{array}\mid t\right)$,   Then the finite field analogue of this series can be considered as 
           ${_2F_1}\left(\begin{array}{cc}
            \chi_4, & \chi_4^3 \\
              & \varphi
           \end{array}\mid t\right)$
 Using the transformations \cite[Lemma 3.3]{mccarthy-pacific}, and \cite[Proposition 2.5]{mccarthy-ffa}, the function ${_2G_2}\left[\begin{array}{cc}
            \frac{1}{4} , & \frac{3}{4}\vspace{1mm} \\
              0, & \frac{1}{2}
           \end{array}\mid \frac{1}{t}\right]$ can be described as $p$-adic analogue of the classical hypergeometric series
 ${_2F_1}\left(\begin{array}{cc}
           \frac{1}{4} , & \frac{3}{4}\vspace{1mm} \\
              & \frac{1}{2}
           \end{array}\mid t\right)$. 
We know that classical hypergeometric series satisfy many powerful identities. 
For example, Gauss \cite{gauss}, Kummer \cite{kummer}, Whipple \cite[p. 54]{lidl}, Saalch\"{u}tz \cite[p. 49]{slater}, Dixon \cite[p. 51]{slater}, and Watson \cite[p. 54]{slater}
studied special values of classical hypergeometric series. For instance, the follwing evaluation of classical hypergeometric series in terms of quotients of classical gamma function was due to Gauss \cite{gauss}. If $R(c-a-b)>0$ then
\begin{align}\label{gauss}
{_2F_1}\left(\begin{array}{cc}
           a, & b \\
           & c
         \end{array}\mid1\right)=\frac{\Gamma(c)\Gamma(c-a-b)}{\Gamma(c-a)\Gamma(c-b)}. 
\end{align}
 If we put $a=\frac{1}{4},$ $b=\frac{3}{4}$ and  $c=1+\frac{1}{2}$ in \eqref{gauss} then we have
\begin{align}\label{gauss-value}
{_2F}_{1}\left(\begin{array}{cc}
           \frac{1}{4}, & \frac{3}{4} \\
           & 1+\frac{1}{2}
         \end{array}\mid1\right)=
         \frac{\Gamma(1+\frac{1}{2})\Gamma(\frac{1}{2})}{\Gamma(1+\frac{1}{4})\Gamma(\frac{3}{4})}.
         \end{align}
         Also, consider the Kummer's theorem \cite{kummer}
         \begin{align}\label{kummer}
         {_2F_1}\left(\begin{array}{cc}
           a, & b \\
           & 1+b-a
         \end{array}\mid-1\right)=\frac{\Gamma(1+b-a)\Gamma(1+\frac{b}{2})}{\Gamma(1+b)\Gamma(1+\frac{b}{2}-a)}.
         \end{align}
          Putting $a=\frac{1}{4}$ and $b=\frac{3}{4}$ into \eqref{kummer} we have    
\begin{align}\label{kummer-value}
{_2F}_{1}\left(\begin{array}{cc}
           \frac{1}{4}, & \frac{3}{4} \\
           & 1+\frac{1}{2}
         \end{array}\mid-1\right)=\frac{\Gamma(1+\frac{1}{2})\Gamma(1+\frac{3}{8})}{\Gamma(1+\frac{3}{4})
         \Gamma(1+\frac{1}{8})}.
\end{align}    
Classical hypergeometric series with dihedral monodromy group can be expressed as elementary functions as their hypergeometric equations can be reformulated to Fuchsian equations with cyclic monodromy groups. For example, two interesting cases that can be expressed as square roots inside powers are:
\begin{align}\label{dg-1}
{_2F_1}\left(\begin{array}{cc}
\frac{a}{2}, & \frac{a+1}{2}\\
~& a+1
\end{array}\mid z\right)=\left(\frac{1+\sqrt{1-z}}{2}\right)^{-a},
\end{align}   
and 
\begin{align}\label{dg-2}
{_2F_1}\left(\begin{array}{cc}
\frac{a}{2}, & \frac{a+1}{2}\vspace{1mm}\\
~& \frac{1}{2}
\end{array}\mid z\right)=\frac{(1-\sqrt{z})^{-a}+(1+\sqrt{z})^{-a}}{2}.
\end{align}      
All these evaluations of Gauss, Kummer etc. motivates us to study the special values of $p$-adic hypergeometric series ${_2G_2}[\cdots].$ Indeed, we completely determine all the possible zeros and non zero values of a certain  family of ${_2G_2}[\cdots].$ We first discuss a theorem that classify all the zeros and non zero values of the function $_2G_2[\cdots].$
           For brevity we write $a\neq\square$ if $a$ is not square in $\mathbb{F}_p.$
\begin{theorem}\label{SV-1}
Let $p\geq3$ be a prime and $t\in\mathbb{F}_p^\times.$  Then we have the following values.
\begin{enumerate}
\item \begin{align}\label{value-1}
{_2G_2}\left[\begin{array}{cc}
\frac{1}{4}, & \frac{3}{4}\vspace{1mm}\\
0, & \frac{1}{2}
\end{array}\mid 1\right]=\frac{-\varphi(-1)}{\Gamma_p(\frac{1}{4})\Gamma_p(\frac{3}{4})}.
\end{align}
\item Let $t\neq1$ and $\frac{t-1}{t}$ be a square in $\mathbb{F}_p^\times$ such that 
$\frac{t-1}{t}=a^2$ for some $a\in\mathbb{F}_p^\times.$ Then we have
\begin{align}\label{value-2}
{_2G_2}\left[\begin{array}{cc}
\frac{1}{4}, & \frac{3}{4}\vspace{1mm}\\
0, & \frac{1}{2}
\end{array}\mid t\right]=\frac{-\varphi(-1)}{\Gamma_p(\frac{1}{4})\Gamma_p(\frac{3}{4})}
(\varphi(1+a)+\varphi(1-a)).
\end{align}
\item Also, if $\frac{t-1}{t}\neq\square$ in $\mathbb{F}_p$ then 
\begin{align}\label{value-3}
{_2G_2}\left[\begin{array}{cc}
\frac{1}{4}, & \frac{3}{4}\vspace{1mm}\\
0, & \frac{1}{2}
\end{array}\mid t\right]=0.
\end{align}
\end{enumerate}
\end{theorem}
\begin{remark}
Note that \eqref{value-1} can be described as a $p$-adic analogue of \eqref{gauss-value}. Theorem \ref{SV-1} provides $p$-adic analogue of \eqref{kummer-value}. The value of the function ${_2G_2}\left[\begin{array}{cc}
\frac{1}{4}, & \frac{3}{4}\\
0, & \frac{1}{2}
\end{array}\mid -1\right]$ completely depends on the prime as if $p\equiv\pm3\pmod{8}$ then it will be equal to zero. However, \eqref{kummer-value} cannot be equal to zero.
Theorem \ref{SV-1} can also be described as $p$-adic analogue of \eqref{dg-1} and \eqref{dg-2} for $a=\frac{1}{2}.$
\end{remark}

Another purpose of this paper is to establish $p$-adic analogous of the Kummer's linear transformation 
 \cite[p. 4 eq. (1)]{bailey} given below. 
\begin{align}\label{kummer-transformation}
&{_2F_1}\left(\begin{array}{cc}
           a, & b \\
           & c
         \end{array}\mid z\right)=\frac{\Gamma(c)\Gamma(c-a-b)}{\Gamma(c-a)\Gamma(c-b)}\cdot
         {_2F_1}\left(\begin{array}{cc}
           a, & b \\
           & 1+a+b-c
         \end{array}\mid1-z\right)\\
         &+\frac{\Gamma(c)\Gamma(a+b-c)}{\Gamma(a)\Gamma(b)}(1-z)^{c-a-b}\cdot
         {_2F_1}\left(\begin{array}{cc}
           c-a, & c-b \\
           & 1+c-a-b
         \end{array}\mid1-z\right).\notag
\end{align} 
The next theorem provides transformations of $p$-adic hypergeometric series which can be described as $p$-adic analogue of  a particular case of Kummer's linear transformation \eqref{kummer-transformation}. This theorem is obtained as a  consequence of Theorem \ref{SV-1}.
\begin{theorem}\label{kummer-1}
Let $p\geq3$ be a prime and $x\in\mathbb{F}_p$ be such that $x\neq0,1$. Then we have the followings:
\begin{enumerate}
\item If $x$ and $1-x$ are not squares in $\mathbb{F}_p$ then
\begin{align}\label{trans-1}
{_2G_2}\left[\begin{array}{cc}
\frac{1}{4}, & \frac{3}{4}\vspace{1mm}\\
0, & \frac{1}{2}
\end{array}\mid\frac{1}{x}\right]={_2G_2}\left[\begin{array}{cc}
\frac{1}{4}, & \frac{3}{4}\vspace{1mm}\\
0, & \frac{1}{2}
\end{array}\mid\frac{1}{1-x}\right].
\end{align}
\item If $x=b^2$ for some $b\in\mathbb{F}_p$ and $1-x$ is not a square in $\mathbb{F}_p$ then 
\begin{align}\label{trans-2}
{_2G_2}\left[\begin{array}{cc}
\frac{1}{4}, & \frac{3}{4}\vspace{1mm}\\
0, & \frac{1}{2}
\end{array}\mid\frac{1}{x}\right]&=\Gamma_p\left(\frac{1}{4}\right)
\Gamma_p\left(\frac{3}{4}\right)
{_2G_2}\left[\begin{array}{cc}
\frac{1}{4}, & \frac{3}{4}\vspace{1mm}\\
0, & \frac{1}{2}
\end{array}\mid\frac{1}{1-x}\right]\\
&+\varphi(-1)(\varphi(1+b)+\varphi(1-b))\notag.
\end{align}
\item If $x,$ and $1-x$ are both squares such that $1-x=a^2$ and $x=b^2$ for some $a,b\in\mathbb{F}_p$ then
\begin{align}\label{trans-3}
&(\varphi(1+b)+\varphi(1-b)){_2G_2}\left[\begin{array}{cc}
\frac{1}{4}, & \frac{3}{4}\vspace{1mm}\\
0, & \frac{1}{2}
\end{array}\mid\frac{1}{x}\right]\notag\\
&=(\varphi(1+a)+\varphi(1-a)){_2G_2}\left[\begin{array}{cc}
\frac{1}{4}, & \frac{3}{4}\vspace{1mm}\\
0, & \frac{1}{2}
\end{array}\mid\frac{1}{1-x}\right].
\end{align} 
\item If $1-x=a^2$ for some $a\in\mathbb{F}_p$ and $x$ is not a square in $\mathbb{F}_p$ then
\begin{align}\label{trans-4}
&\Gamma_p\left(\frac{1}{4}\right)
\Gamma_p\left(\frac{3}{4}\right){_2G_2}\left[\begin{array}{cc}
\frac{1}{4}, & \frac{3}{4}\vspace{1mm}\\
0, & \frac{1}{2}
\end{array}\mid\frac{1}{x}\right]\notag\\
&={_2G_2}\left[\begin{array}{cc}
\frac{1}{4}, & \frac{3}{4}\vspace{1mm}\\
0, & \frac{1}{2}
\end{array}\mid\frac{1}{1-x}\right]
-\varphi(-1)(\varphi(1+a)+\varphi(1-a)).
\end{align}
\end{enumerate} 
\end{theorem}
Note that the finite field analogue of Theorem \ref{kummer-1} involving characters of order 4 follows from
Greene's evaluation \cite[Theorem 4.4 (i)]{greene} and if we use this result of Greene along with the relations \cite[Lemma 3.3]{mccarthy-pacific}, and \cite[Proposition 2.5]{mccarthy-ffa} then we also obtain a similar transformation for the $p$-adic hypergeometric series ${_2G_2}\left[\begin{array}{cc}
\frac{1}{4}, & \frac{3}{4}\vspace{1mm}\\
0, & \frac{1}{2}
\end{array}\mid t\right]$ under the condition that $p\equiv1\pmod{4}.$ However, Theorem \ref{kummer-1} has no congruence condition on primes.
\par Fuselier-McCarthy \cite{fm} evaluated certain  summation identities for $p$-adic hypergeometric series. This motivates us to study summation identities of $p$-adic hypergeometric series.
\begin{theorem}\label{sum-1}
Let $p\geq3$ be a prime. Let $x\in\mathbb{F}_p^{\times}.$ Then we have the following:
\begin{enumerate}
\item \begin{align}\label{formula-1}
\sum_{t\in\mathbb{F}_p^\times}\varphi(t(t-1))~{_2G_2}\left[\begin{array}{cc}
\frac{1}{4}, & \frac{3}{4}\vspace{1mm}\\
0, & 0
\end{array}\mid t\right]=-1+\frac{p\varphi(-1)}{\Gamma_p(\frac{1}{4})\Gamma_p(\frac{3}{4})}.
\end{align}
\item If $x\neq1$ and $1-x$ is not a square in $\mathbb{F}_p$ then we have
\begin{align}\label{formula-2}
\sum_{t\in\mathbb{F}_p^\times}\varphi(t(t-1))~{_2G_2}\left[\begin{array}{cc}
\frac{1}{4}, & \frac{3}{4}\vspace{1mm}\\
0, & 0
\end{array}\mid\frac{t}{x}\right]=-1.
\end{align}
\item If $x\neq1$ and $1-x=a^2$ for some $a\in\mathbb{F}_p$ then
\begin{align}\label{formula-3}
\sum_{t\in\mathbb{F}_p^\times}\varphi(t(t-1))~{_2G_2}\left[\begin{array}{cc}
\frac{1}{4}, & \frac{3}{4}\vspace{1mm}\\
0, & 0
\end{array}\mid\frac{t}{x}\right]=-1+\frac{p\varphi(-1)(\varphi(1+a)+\varphi(1-a))}
{\Gamma_p(\frac{1}{4})\Gamma_p(\frac{3}{4})}.
\end{align}
\end{enumerate}
\end{theorem}
The following theorem gives a summation identity of Gaussian hypergeometric functions involving characters of order 4.
\begin{theorem}\label{sum-3}
Let $p\geq3$ be a prime such that $p\equiv1\pmod{4}$. Let $x\in\mathbb{F}_p^{\times}$ and $\chi_4$ be a multiplicative character of $\mathbb{F}_p$ of order $4.$ Then we have the following.
\begin{enumerate}
\item \begin{align}\label{formula-4}
\sum_{t\in\mathbb{F}_p^\times}\varphi(1-t)~{_2F_1}\left(\begin{array}{cc}
\chi_4, & \chi_4^3\\
 & \varepsilon
\end{array}\mid t\right)=\frac{1}{p}+\chi_4(-1).
\end{align}
\item If $x\neq1$ and $1-x$ is not a square in $\mathbb{F}_p$ then we have
\begin{align}\label{formula-5}
\sum_{t\in\mathbb{F}_p^\times}\varphi(x-t)~{_2F_1}\left(\begin{array}{cc}
\chi_4, & \chi_4^3\\
 & \varepsilon
\end{array}\mid t\right)=\frac{\varphi(x)}{p}.
\end{align}
\item If $x\neq1$ and $1-x=a^2$ for some $a\in\mathbb{F}_p^{\times}$ then
\begin{align}\label{formula-6}
\sum_{t\in\mathbb{F}_p^\times}\varphi(x-t)~{_2F_1}\left(\begin{array}{cc}
\chi_4, & \chi_4^3\\
 & \varepsilon
\end{array}\mid t\right)=\frac{\varphi(x)}{p}+\varphi(x)\chi_4(-1)(\varphi(1+a)+\varphi(1-a)).
\end{align}
\end{enumerate}
\end{theorem}

Apart from Kummer's transformation there are other interesting transformation formulas exist in the literature. For example, Euler \cite[p. 10]{slater}, Whipple \cite{whipple}, Dixon \cite{dixon} studied transformation properties of classical hypergeometric series. However, we are interested in the Pfaff's transformation \cite[p. 31]{slater} 
\begin{align}
{_2F_1}\left(\begin{array}{cc}
a, & b\\
~ & c
\end{array}\mid z\right)=(1-z)^{-a}{_2F_1}\left(\begin{array}{cc}
a, & c-b\\
~ & c
\end{array}\mid \frac{z}{z-1}\right).\notag
\end{align}
In particular, if $a=\frac{1}{4}$, $b=\frac{3}{4}$, and $c=\frac{1}{2}$ then the above result gives
\begin{align}\label{pfaff}
{_2F_1}\left(\begin{array}{cc}
\frac{1}{4}, & \frac{3}{4}\vspace{1mm}\\
~ & \frac{1}{2}
\end{array}\mid z\right)=(1-z)^{-1/4}{_2F_1}\left(\begin{array}{cc}
\frac{1}{4}, & \frac{-1}{4}\\
~ & \frac{1}{2}
\end{array}\mid \frac{z}{z-1}\right).
\end{align}
We know that $p$-adic analogue of ${_2F_1}\left(\begin{array}{cc}
\frac{1}{4}, & \frac{-1}{4}\vspace{1mm}\\
~ & \frac{1}{2}
\end{array}\mid z\right)$ can be described as the function ${_2G_2}\left[\begin{array}{cc}
\frac{1}{4}, & \frac{-1}{4}\vspace{1mm}\\
0, & \frac{1}{2}
\end{array}\mid\frac{1}{z}\right]={_2G_2}\left[\begin{array}{cc}
\frac{1}{4}, & 1-\frac{1}{4}\vspace{1mm}\\
0, & \frac{1}{2}
\end{array}\mid\frac{1}{z}\right].$ Then a $p$-adic analogue of Pfaff's transformation \eqref{pfaff} is described in the next theorem. 
\begin{theorem}\label{paf}
Let $p\geq3$ be a prime and $1\neq x\in\mathbb{F}_p^\times$. Then we have the followings.
\begin{enumerate}
\item If $1-x\neq\square$ then 
\begin{align}\label{paf-1}
{_2G_2}\left[\begin{array}{cc}
\frac{1}{4}, & \frac{3}{4}\vspace{1mm}\\
0, & \frac{1}{2}
\end{array}\mid\frac{1}{x}\right]={_2G_2}\left[\begin{array}{cc}
\frac{1}{4}, & \frac{3}{4}\vspace{1mm}\\
0, & \frac{1}{2}
\end{array}\mid\frac{x-1}{x}\right].
\end{align} 
\item If $1-x=a^2$ for some $a\in\mathbb{F}_p^\times$ then 
\begin{align}\label{paf-2}
&\varphi(a)(\varphi(a+1)+\varphi(a-1)){_2G_2}\left[\begin{array}{cc}
\frac{1}{4}, & \frac{3}{4}\vspace{1mm}\\
0, & \frac{1}{2}
\end{array}\mid\frac{1}{x}\right]\notag\\
&=(\varphi(1+a)+\varphi(1-a)){_2G_2}\left[\begin{array}{cc}
\frac{1}{4}, & \frac{3}{4}\vspace{1mm}\\
0, & \frac{1}{2}
\end{array}\mid\frac{x-1}{x}\right].
\end{align} 
\end{enumerate}
\end{theorem}

 The rest of this paper is organized as follows. We introduce some basic definitions in Section 2 including Gauss sum and $p$-adic gamma function.  In section 2 we state some results including Hasse-Davenport result, Gross-Koblitz formula. We give the proofs of the main theorems in Section 3.

\section{Notation and Preliminary results}
Let $\overline{\mathbb{Q}_p}$ be the algebraic closure of $\mathbb{Q}_p$ and $\mathbb{C}_p$ be the completion of $\overline{\mathbb{Q}_p}.$ Since each $\chi\in\mathbb{F}_p^\times$ takes values from $\mu_{p-1},$ the group of $(p-1)$-th roots of unity in $\mathbb{C}^\times,$ and  $\mathbb{Z}_p^\times$ contains $\mu_{p-1}$, so   we may assume that the multiplicative characters of $\mathbb{F}_p^\times$ to be mapped 
$\chi:\mathbb{F}_p^{\times}\mapsto\mathbb{Z}_p^\times.$
Recall that $\omega:\mathbb{F}_p^\times\mapsto\mathbb{Z}_p^\times$ is the Teichm$\ddot{u}$ller character. Also,  $\widehat{\mathbb{F}_p^\times}=\{\omega^j:0\leq j\leq p-2\}$ and  $\overline{\omega}$ denotes the inverse of $\omega.$ 

\subsection{Preliminary results on Multiplicative characters and Gauss sums:}
The following result gives the orthogonality relation of multiplicative characters.
\begin{lemma}\cite[Chapter 8]{ireland}
Let $p$ be an odd prime. Then we have
\begin{align}\label{orthogonal-1}
\sum_{\chi\in\widehat{\mathbb{F}_p^\times}}\chi(x)=\left\{
   \begin{array}{ll}
    p-1 , & \hbox{if $x=1$;} \\
  0, & \hbox{if $x\neq1$.}
   \end{array}
 \right.
\end{align}
\end{lemma}
\par Let $\zeta_p$ be a fixed primitive $p$-th root of unity in $\overline{\mathbb{Q}_p}.$ 
For $\chi \in \widehat{\mathbb{F}_p^\times}$, the Gauss sum is defined by
\begin{align}
g(\chi):=\sum\limits_{x\in \mathbb{F}_p}\chi(x)~\zeta_p^x.\notag
\end{align}
From the definition we can say that $g(\varepsilon)=-1$. For more details on Gauss sums, see \cite{berndt}. We now introduce some properties of Gauss sums. Let 
$\delta: \widehat{\mathbb{F}_p^\times}
\rightarrow\{0,1\}$ be defined by
\begin{align}
\delta(\chi)=\left\{
   \begin{array}{ll}
    1 , & \hbox{if $\chi=\varepsilon$;} \\
  0, & \hbox{if $\chi\neq\varepsilon$.}
   \end{array}
 \right.
\end{align}
We start with a result that provides a formula for the multiplicative inverse of Gauss sum.
\begin{lemma}\cite[eq. 1.12]{greene} Let $\chi\in \widehat{\mathbb{F}_p^\times}.$ Then
\begin{align}\label{inverse}
g(\chi)g(\overline{\chi})=p\chi(-1)-(p-1)\delta(\chi).
\end{align}
\end{lemma}
Another important product formula for Gauss sums is the Hasse-Davenport formula.
\begin{theorem}\cite[Hasse-Davenport relation, Theorem 11.3.5]{berndt}
Let $\chi$ be a character of order $m$ of $\widehat{\mathbb{F}_p^\times}$ for some positive integer $m.$ For a multiplicative character $\psi$ of $\widehat{\mathbb{F}_p^\times}$ we  have
\begin{align}\label{dh}
\prod_{i=0}^mg(\psi\chi^i)=g(\psi^m)\psi^{-m}(m)\prod_{i=1}^{m-1}g(\chi^i).
\end{align}
\end{theorem}
The following lemma relates Gauss and Jacobi sums.
\begin{lemma}\cite[eq. 1.14]{greene} 
Let $\chi_1,\chi_2\in \widehat{\mathbb{F}_p^\times}.$ Then
\begin{align}\label{gauss-jacobi}
J(\chi_1,\chi_2)=\frac{g(\chi_1)g(\chi_2)}{g(\chi_1\chi_2)}+(p-1)\chi_2(-1)\delta(\chi_1\chi_2).
\end{align}
\end{lemma}
Let $\chi,\psi$ be multiplicative characters of $\mathbb{F}_p.$ Then the following special values of binomials are very useful to prove our main results, for more details we refer \cite[eq. 2.12, eq. 2.7]{greene}.
\begin{align}\label{rel-1}
{\chi\choose\varepsilon}={\chi\choose\chi}=-\frac{1}{p}+\frac{p-1}{p}\delta(\chi).
\end{align}
\begin{align}\label{rel-2}
{\chi\choose\psi}={\psi\overline{\chi}\choose\psi}\psi(-1).
\end{align}

\subsection{$p$-adic Preliminaries:}
We recall the $p$-adic gamma function, for further details see \cite{kob}.
For a positive integer $n$,
the $p$-adic gamma function $\Gamma_p(n)$ is defined as
\begin{align}
\Gamma_p(n):=(-1)^n\prod\limits_{0<j<n,p\nmid j}j\notag
\end{align}
and one can extend it to all $x\in\mathbb{Z}_p$ by setting $\Gamma_p(0):=1$ and
\begin{align}
\Gamma_p(x):=\lim_{x_n\rightarrow x}\Gamma_p(x_n)\notag
\end{align}
for $x\neq0$, where $x_n$ runs through any sequence of positive integers $p$-adically approaching $x$.
Two important product formulas of $p$-adic gamma function form \cite{gross} are as follows.
If $x\in\mathbb{Z}_p$ then 
\begin{align}\label{prod-3}
\Gamma_p(x)\Gamma_p(1-x)=(-1)^{a_0(x)},
\end{align}
where $a_0(x)\equiv x\pmod{p}$ such that $a_0(x)\in\{1,2,\ldots,p\}.$
If $m\in\mathbb{Z}^+,$ 
$p\nmid m$  and $x=\frac{r}{p-1}$ with $0\leq r\leq p-1$ then
\begin{align}\label{prod-1}
\prod_{h=0}^{m-1}\Gamma_p\left(\frac{x+h}{m}\right)=\omega(m^{(1-x)(1-p)})~\Gamma_p(x)\prod_{h=1}^{m-1}\Gamma_p\left(\frac{h}{m}\right).
\end{align}
Another interesting product formula of $p$-adic gamma function given in \cite{mccarthy-pacific} as a consequence of \eqref{prod-1} described as follows. Let $t\in\mathbb{Z}^{+}$ and $p\nmid t$. Then for $0\leq j\leq p-2$ we have
\begin{align}\label{prod-2}
\omega(t^{-tj})\Gamma_p\left(\left\langle\frac{-tj}{p-1}\right\rangle\right)\prod_{h=1}^{t-1}\Gamma_p\left(\frac{h}{t}\right)
=\prod_{h=1}^{t}\Gamma_p\left(\left\langle\frac{h}{t}-\frac{j}{p-1}\right\rangle\right).
\end{align}
Let $\pi \in \mathbb{C}_p$ be the fixed root of the polynomial $x^{p-1} + p$, which satisfies the congruence condition
$\pi \equiv \zeta_p-1 \pmod{(\zeta_p-1)^2}$. The Gross-Koblitz formula relates the Gauss sum and $p$-adic gamma function as follows.
\begin{theorem}\cite[Gross-Koblitz formula]{gross}\label{gross-koblitz} For $j\in \mathbb{Z}$,
\begin{align}
g(\overline{\omega}^j)=-\pi^{(p-1)\langle\frac{j}{p-1} \rangle}\Gamma_p\left(\left\langle \frac{j}{p-1} \right\rangle\right).\notag
\end{align}
\end{theorem}
The next two lemmas are helpful in the proof of our main results. These two lemmas are direct applications of Gross-Koblitz formula.
\begin{lemma}
For $1\leq j\leq p-2$
\begin{align}\label{lemma-1}
\Gamma_p\left(\left\langle1-\frac{j}{p-1}\right\rangle\right)\Gamma_p\left(\left\langle\frac{j}{p-1}\right\rangle\right)
=-\omega^j(-1).
\end{align}
\end{lemma}
\begin{proof}
Applying Gross-Koblitz formula (Theorem \ref{gross-koblitz}) on the left hand side of \eqref{lemma-1} and then using \eqref{inverse} it is straightforward to verify the lemma.
\end{proof}
\begin{lemma}
For $1\leq j\leq p-2$ we have
\begin{align}\label{lemma-2}
&\frac{(-p)^{-\lfloor\frac{1}{2}+\frac{j}{p-1}\rfloor}}{\Gamma_p(\frac{1}{2})}~
\Gamma_p\left(\left\langle\frac{1}{2}+\frac{j}{p-1}\right\rangle\right)
\Gamma_p\left(\left\langle1-\frac{j}{p-1}\right\rangle\right)\notag\\
&=\frac{1}{p}\sum_{t\in\mathbb{F}_p^\times}
\overline{\omega}^j(-t)\varphi(t(t-1)).
\end{align}
\end{lemma}
\begin{proof}
Let $U=\frac{(-p)^{-\lfloor\frac{1}{2}+\frac{j}{p-1}\rfloor}}{\Gamma_p(\frac{1}{2})}~
\Gamma_p\left(\left\langle\frac{1}{2}+\frac{j}{p-1}\right\rangle\right)
\Gamma_p\left(\left\langle1-\frac{j}{p-1}\right\rangle\right).$ Using Gross-Koblitz formula (Theorem \ref{gross-koblitz}), \eqref{inverse}, and \eqref{gauss-jacobi} we obtain
 \begin{align*}
 U&=\frac{\varphi\omega^j(-1)}{p}J(\varphi\overline{\omega}^j,\varphi)
 =\frac{1}{p}\sum_{t\in\mathbb{F}_p^\times}\overline{\omega}^j(-t)\varphi(t(t-1)).
 \end{align*}
 This completes the proof of the lemma.
 \end{proof}
\section{Proof of the theorems}
We begin with a proposition which explicitly determines the value of a character sum. We use this proposition to prove Theorem \ref{SV-1} and Theorem \ref{sum-1}.
\begin{proposition}\label{prop-1}
For $x\in\mathbb{F}_p^{\times}$ we have
\begin{align*}
&\sum_{\chi\in\widehat{\mathbb{F}_p^\times}}g(\varphi\chi^2)g(\varphi\overline{\chi})g(\overline{\chi})
\chi\left(\frac{x}{4}\right)\\
&=\left\{
   \begin{array}{ll}
    0 , & \hbox{if $1-x\neq\square$;} \\
    p(p-1)\varphi(-2),  & \hbox{if $x=1$;}\\
  p(p-1)\varphi(-2)(\varphi(1+a)+\varphi(1-a)), & \hbox{if $x\neq1$ and $1-x=a^2$.}
   \end{array}
 \right.
\end{align*}
\end{proposition}
\begin{proof} 
Let $A=\displaystyle\sum_{\chi\in\widehat{\mathbb{F}_p^\times}}g(\varphi\chi^2)g(\varphi\overline{\chi})g(\overline{\chi})
\chi\left(\frac{x}{4}\right).$ Multiplying both numerator and denominator by $g(\varphi\chi)$ we can write  
\begin{align}\label{eqn-1}
A=\sum_{\chi\in\widehat{\mathbb{F}_q^\times}}\frac{g(\varphi\chi^2)g(\overline{\chi})}{g(\varphi\chi)}
~g(\varphi\chi)g(\varphi\overline{\chi})~\chi\left(\frac{x}{4}\right).
\end{align}
Applying \eqref{gauss-jacobi} we have
\begin{align}\label{eqn-2}
\frac{g(\varphi\chi^2)g(\overline{\chi})}{g(\varphi\chi)}=J(\varphi\chi^2,\overline{\chi})-(p-1)\chi(-1)\delta(\varphi\chi).
\end{align}
Also, applying \eqref{inverse} we have
\begin{align}\label{eqn-3}
g(\varphi\chi)g(\varphi\overline{\chi})=p\varphi\chi(-1)-(p-1)\delta(\varphi\chi).
\end{align}
Substituting \eqref{eqn-2} and \eqref{eqn-3} into \eqref{eqn-1} we obtain
\begin{align}\label{eqn-4}
A=A_1+A_2+A_3+A_4,
\end{align}
where \begin{align}\label{eqn-5}
A_1=p\varphi(-1)\sum_{\chi\in\widehat{\mathbb{F}_p^\times}}J(\varphi\chi^2,\overline{\chi})\chi\left(\frac{-x}{4}\right),
\end{align}
\begin{align}\label{eqn-6}
A_2&=-(p-1)\sum_{\chi\in\widehat{\mathbb{F}_p^\times}}\delta(\varphi\chi)J(\varphi\chi^2,\overline{\chi})
~\chi\left(\frac{x}{4}\right)\notag\\
&=-(p-1)\varphi(x)J(\varphi,\varphi),
\end{align}
\begin{align}\label{eqn-7}
A_3&=-p(p-1)\varphi(-1)\sum_{\chi\in\widehat{\mathbb{F}_p^\times}}\delta(\varphi\chi)~\chi\left(\frac{x}{4}\right)\notag\\
&=-p(p-1)\varphi(-x),
\end{align}
and 
\begin{align}\label{eqn-8}
A_4&=(p-1)^2\sum_{\chi\in\widehat{\mathbb{F}_p^\times}}\delta(\varphi\chi)~\chi\left(\frac{-x}{4}\right)\notag\\
&=(p-1)^2\varphi(-x).
\end{align}
Using \eqref{binomial}, and \eqref{rel-1} in \eqref{eqn-6} we have
\begin{align}\label{eqn-9}
A_2&=-p(p-1)\varphi(-x){\varphi\choose\varphi}\notag\\
&=(p-1)\varphi(-x).
\end{align}
Adding \eqref{eqn-7}, \eqref{eqn-8}, and \eqref{eqn-9} we obtain
\begin{align}\label{eqn-10}
A_2+A_3+A_4=0.
\end{align}
Substituting \eqref{eqn-10} into \eqref{eqn-4} we have
\begin{align}
A=A_1=p\varphi(-1)\sum_{\chi\in\widehat{\mathbb{F}_p^\times}}J(\varphi\chi^2,\overline{\chi})~\chi\left(\frac{-x}{4}\right).\notag
\end{align}
\eqref{binomial} gives 
\begin{align}
A=p^2\varphi(-1)\sum_{\chi\in\widehat{\mathbb{F}_p^\times}}{\varphi\chi^2\choose\chi}~
\chi\left(\frac{x}{4}\right).\notag
\end{align} 
If we use \eqref{rel-2} then we have ${\varphi\chi^2\choose\chi}=\chi(-1){\varphi\overline{\chi}\choose\chi}.$ This yields
\begin{align}\label{eqn-11}
A&=p^2\varphi(-1)\sum_{\chi\in\widehat{\mathbb{F}_p^\times}}{\varphi\overline{\chi}\choose\chi}~\chi\left(\frac{-x}{4}\right)\notag\\
&=p\varphi(-1)\sum_{\substack{\chi\in\widehat{\mathbb{F}_p^\times}\\y\in\mathbb{F}_p}}\varphi(y)\overline{\chi}(y)
\overline{\chi}(1-y)\chi\left(\frac{x}{4}\right).
\end{align}
Replacing $\chi$ by $\overline{\chi}$ in \eqref{eqn-11} we obtain
\begin{align}\label{eqn-13}
A=p\varphi(-1)\sum_{y\in\mathbb{F}_p}\varphi(y)\sum_{\chi\in\widehat{\mathbb{F}_p^\times}}
\chi\left(\frac{4y(1-y)}{x}\right).
\end{align}
Using the orthogonality relation \eqref{orthogonal-1} we can say that second sum present in \eqref{eqn-13} is non zero  if and only if the equation $4y^2-4y+x=0$ admits a solution. We know that $4y^2-4y+x=0$ is solvable if and only if $1-x$ is a square in $\mathbb{F}_p.$ Let $1-x=a^2$ for some $a\in\mathbb{F}_p.$ Then $\frac{1}{2}(1\pm a)$ are solutions of
$4y^2-4y+x=0.$ Hence, we obtain
\begin{align*}
A=\left\{
   \begin{array}{ll}
   p(p-1)\varphi(-2), & \hbox{if $x=1$;} \\
    p(p-1)\varphi(-2)(\varphi(1+a)+\varphi(1-a)) , & \hbox{if $x\neq1$ and $1-x=a^2$;} \\
   0, & \hbox{if $1-x\neq\square$.}
   \end{array}
 \right.
\end{align*} 
This completes the proof.
\end{proof}
In the next proposition, we again consider the same character sum as considered in Proposition \ref{prop-1} and express the sum as a special value of $p$-adic hypergeometric series. We use this proposition to prove Theorem \ref{SV-1}.
\begin{proposition}\label{prop-2}
For $x\in\mathbb{F}_p^{\times}$ we have
\begin{align*}
&\sum_{\chi\in\widehat{\mathbb{F}_p^\times}}g(\varphi\chi^2)g(\varphi\overline{\chi})g(\overline{\chi})
\chi\left(\frac{x}{4}\right)=p(1-p)\varphi(2)\Gamma_p\left(\frac{1}{4}\right)\Gamma_p\left(\frac{3}{4}\right)~{_2G_2}\left[\begin{array}{cc}
\frac{1}{4},& \frac{3}{4}\vspace{1mm}\\
0,& \frac{1}{2}
\end{array}\mid\frac{1}{x}\right].
\end{align*}
\end{proposition}
\begin{proof}
Let $A=\displaystyle\sum_{\chi\in\widehat{\mathbb{F}_p^\times}}g(\varphi\chi^2)g(\varphi\overline{\chi})g(\overline{\chi})
\chi\left(\frac{x}{4}\right).$ Since $\widehat{\mathbb{F}_p^\times}=\{\omega^j:0\leq j\leq p-2\},$ replacing $\chi$ by 
$\omega^j$ and applying Gross-Koblitz formula we obtain
\begin{align}\label{eqn-14}
A&=-\sum_{j=0}^{p-2}\omega^j\left(\frac{x}{4}\right)\pi^{(p-1)\ell_j}
~\Gamma_p\left(\left\langle\frac{1}{2}-\frac{2j}{p-1}\right\rangle\right)\Gamma_p\left(\left\langle\frac{1}{2}+\frac{j}{p-1}\right\rangle\right)\\
&\times\Gamma_p\left(\frac{j}{p-1}\right)\notag,
\end{align}
where $\ell_j=\langle\frac{1}{2}-\frac{2j}{p-1}\rangle+\langle\frac{1}{2}+\frac{j}{p-1}\rangle+\left(\frac{j}{p-1}\right).$
Applying \eqref{prod-1} with $x=\left\langle\frac{1}{2}-\frac{2j}{p-1}\right\rangle$ and $m=2$ we obtain

\begin{align*}
\Gamma_p\left(\left\langle\frac{1}{2}-\frac{2j}{p-1}\right\rangle\right)
=\frac{\Gamma_p\left(\frac{1}{2}\left\langle\frac{1}{2}-\frac{2j}{p-1}\right\rangle\right)
\Gamma_p\left(\frac{1}{2}\left\langle\frac{1}{2}-\frac{2j}{p-1}\right\rangle+\frac{1}{2}\right)}
{\Gamma_p(\frac{1}{2})~\omega(2^{(1-p)(1-\langle\frac{1}{2}-\frac{2j}{p-1}\rangle)})}.
\end{align*}
Taking $j$ in the intervals $[0,\lfloor\frac{p-1}{4}\rfloor], (\lfloor\frac{p-1}{4}\rfloor, \lfloor\frac{3(p-1)}{4}\rfloor]$ and 
$(\lfloor\frac{3(p-1)}{4}\rfloor,p-2]$ we verify that
\begin{align*}
&\Gamma_p\left(\frac{1}{2}\left\langle\frac{1}{2}-\frac{2j}{p-1}\right\rangle\right)
\Gamma_p\left(\frac{1}{2}\left\langle\frac{1}{2}-\frac{2j}{p-1}\right\rangle+\frac{1}{2}\right)\\
&=\Gamma_p\left(\left\langle\frac{1}{4}-\frac{j}{p-1}\right\rangle\right)\Gamma_p\left(\left\langle\frac{3}{4}-\frac{j}{p-1}\right\rangle\right)
\end{align*}
and $\omega(2^{(1-p)(1-\langle\frac{1}{2}-\frac{2j}{p-1}\rangle)})
=\varphi(2)~\overline{\omega}^j(4).$ Therefore, we can write
\begin{align}\label{eqn-15}
\Gamma_p\left(\left\langle\frac{1}{2}-\frac{2j}{p-1}\right\rangle\right)
=\frac{\Gamma_p\left(\left\langle\frac{1}{4}-\frac{j}{p-1}\right\rangle\right)\Gamma_p\left(\left\langle\frac{3}{4}-\frac{j}{p-1}\right\rangle\right)}
{\Gamma_p(\frac{1}{2})~\varphi(2)~\overline{\omega}^j(4)}.
\end{align}
Substituting \eqref{eqn-15} into \eqref{eqn-14} we obtain
\begin{align}\label{eqn-16}
A&=-\frac{\varphi(2)}{\Gamma_p(\frac{1}{2})}\sum_{j=0}^{p-2}\omega^j(x)~\pi^{(p-1)\ell_j}~\Gamma_p\left(\left\langle\frac{1}{4}-\frac{j}{p-1}\right\rangle\right)\Gamma_p\left(\left\langle\frac{3}{4}-\frac{j}{p-1}\right\rangle\right)\\
&\times\Gamma_p\left(\left\langle\frac{1}{2}+\frac{j}{p-1}\right\rangle\right)
\Gamma_p\left(\frac{j}{p-1}\right)\notag.
\end{align}
Now, \begin{align}
\ell_j&=\left\langle\frac{1}{2}-\frac{2j}{p-1}\right\rangle+\left\langle\frac{1}{2}+\frac{j}{p-1}\right\rangle+\left(\frac{j}{p-1}\right)\notag\\
&=1-\left\lfloor\frac{1}{2}-\frac{2j}{p-1}\right\rfloor-\left\lfloor\frac{1}{2}+\frac{j}{p-1}\right\rfloor.\notag
\end{align}
By considering $\left\lfloor\frac{1}{2}-\frac{2j}{p-1}\right\rfloor=2k+s$ for some $k\in\mathbb{Z}$ and $s=0,1$ it is straight forward to verify that
$\left\lfloor\frac{1}{2}-\frac{2j}{p-1}\right\rfloor=\left\lfloor\frac{1}{4}-\frac{j}{p-1}\right\rfloor
+\left\lfloor\frac{3}{4}-\frac{j}{p-1}\right\rfloor.$ This gives
\begin{align}\label{eqn-17}
\ell_j=1-\left\lfloor\frac{1}{4}-\frac{j}{p-1}\right\rfloor-\left\lfloor\frac{3}{4}-\frac{j}{p-1}\right\rfloor
-\left\lfloor\frac{1}{2}+\frac{j}{p-1}\right\rfloor.
\end{align}
Substituting \eqref{eqn-17} into \eqref{eqn-16} we obtain
\begin{align*}
A=p(1-p)\varphi(2)\Gamma_p\left(\frac{1}{4}\right)\Gamma_p\left(\frac{3}{4}\right)~{_2G_2}\left[\begin{array}{cc}
\frac{1}{4},& \frac{3}{4}\vspace{1mm}\\
0,& \frac{1}{2}
\end{array}\mid\frac{1}{x}\right].
\end{align*}
This completes the proof.
\end{proof}
\begin{proof}[Proof of Theorem \ref{SV-1}]
Comparing Proposition \ref{prop-1} and Proposition \ref{prop-2} for $x=1$ we prove \eqref{value-1}. Now, letting $x\neq0,1$, and $1-x=a^2$ for some $a\in\mathbb{F}_p^\times$ and then comparing Proposition \ref{prop-1} and Proposition \ref{prop-2}
we obtain
\begin{align}\label{equation-1}
{_2G_2}\left[\begin{array}{cc}
\frac{1}{4},& \frac{3}{4}\vspace{1mm}\\
0,& \frac{1}{2}
\end{array}\mid\frac{1}{x}\right]=\frac{-\varphi(-1)}{\Gamma_p(\frac{1}{4})\Gamma_p(\frac{3}{4})}(\varphi(1+a)+\varphi(1-a)).
\end{align}
Then replacing $x$ by $\frac{1}{t}$ in \eqref{equation-1} we obtain \eqref{value-2}.
Finally, if $1-x$ is not a square in $\mathbb{F}_p$ then again, comparing Proposition \ref{prop-1} and Proposition \ref{prop-2}
we obtain
\begin{align}\label{equation-2}
{_2G_2}\left[\begin{array}{cc}
\frac{1}{4},& \frac{3}{4}\vspace{1mm}\\
0,& \frac{1}{2}
\end{array}\mid\frac{1}{x}\right]=0.
\end{align} Replacing $x$ by $\frac{1}{t}$ in \eqref{equation-2} we derive \eqref{value-3}.
This completes the proof of the theorem.
\end{proof}


\begin{proof}[Proof of Theorem \ref{kummer-1}]
If $t\neq0,1$ and $1-\frac{1}{t}=\frac{t-1}{t}\neq\square$ then from \eqref{value-3} we have
\begin{align}\label{equation-3}
{_2G_2}\left[\begin{array}{cc}
\frac{1}{4},& \frac{3}{4}\vspace{1mm}\\
0,& \frac{1}{2}
\end{array}\mid t\right]=0.
\end{align}
Replacing $t$ by $\frac{1}{x}$ in \eqref{equation-3} we obtain that if $1-x\neq\square$ then 
\begin{align}\label{equation-4}
{_2G_2}\left[\begin{array}{cc}
\frac{1}{4},& \frac{3}{4}\vspace{1mm}\\
0,& \frac{1}{2}
\end{array}\mid \frac{1}{x}\right]=0.
\end{align}
Similarly, if $t\neq0,1$ and $1-\frac{1}{1-t}=\frac{t}{t-1}\neq\square$ then \eqref{value-3} gives
\begin{align}\label{equation-5}
{_2G_2}\left[\begin{array}{cc}
\frac{1}{4},& \frac{3}{4}\vspace{1mm}\\
0,& \frac{1}{2}
\end{array}\mid1- t\right]=0.
\end{align}
Replacing $1-t$ by $\frac{1}{1-x}$ in \eqref{equation-5} we can write that if $x\neq\square$ then
\begin{align}\label{equation-6}
{_2G_2}\left[\begin{array}{cc}
\frac{1}{4},& \frac{3}{4}\vspace{1mm}\\
0,& \frac{1}{2}
\end{array}\mid\frac{1}{1-x}\right]=0.
\end{align}
Combining \eqref{equation-4} and \eqref{equation-6} we obtain \eqref{trans-1}.
Now, let $x=b^2.$ Putting $1-x=\frac{1}{t}$ we have $1-\frac{1}{t}=b^2.$ Applying \eqref{value-2} we have
\begin{align}
{_2G_2}\left[\begin{array}{cc}
\frac{1}{4},& \frac{3}{4}\vspace{1mm}\\
0,& \frac{1}{2}
\end{array}\mid t\right]=\frac{-\varphi(-1)}{\Gamma_p(\frac{1}{4})\Gamma_p(\frac{3}{4})}(\varphi(1+b)+\varphi(1-b)).\notag
\end{align}
Therefore, if $x=b^2$ then
\begin{align}\label{equation-7}
{_2G_2}\left[\begin{array}{cc}
\frac{1}{4},& \frac{3}{4}\vspace{1mm}\\
0,& \frac{1}{2}
\end{array}\mid \frac{1}{1-x}\right]=\frac{-\varphi(-1)}{\Gamma_p(\frac{1}{4})\Gamma_p(\frac{3}{4})}(\varphi(1+b)+\varphi(1-b)).
\end{align}
Combining \eqref{equation-4} and \eqref{equation-7} we deduce \eqref{trans-2}.
Let $1-x=a^2.$ Also, let $x=\frac{1}{t}$. Then $1-\frac{1}{t}=a^2.$ Using \eqref{value-2} we obtain that 
\begin{align}
{_2G_2}\left[\begin{array}{cc}
\frac{1}{4},& \frac{3}{4}\vspace{1mm}\\
0,& \frac{1}{2}
\end{array}\mid t\right]=\frac{-\varphi(-1)}{\Gamma_p(\frac{1}{4})\Gamma_p(\frac{3}{4})}(\varphi(1+a)+\varphi(1-a)).\notag
\end{align} 
Therefore, if $1-x=a^2$ then 
\begin{align}\label{equation-8}
{_2G_2}\left[\begin{array}{cc}
\frac{1}{4},& \frac{3}{4}\vspace{1mm}\\
0,& \frac{1}{2}
\end{array}\mid \frac{1}{x}\right]=\frac{-\varphi(-1)}{\Gamma_p(\frac{1}{4})\Gamma_p(\frac{3}{4})}(\varphi(1+a)+\varphi(1-a)).
\end{align}
Combining \eqref{equation-7} and \eqref{equation-8} we derive \eqref{trans-3}.
Finally, comparing \eqref{equation-6} and \eqref{equation-8} we obtain \eqref{trans-4}. This completes the proof of the theorem.
\end{proof}


\begin{proof}[Proof of Theorem \ref{sum-1}]
Again, consider the sum $A=\displaystyle\sum_{\chi\in\widehat{\mathbb{F}_p^\times}}g(\varphi\chi^2)g(\varphi\overline{\chi})g(\overline{\chi})
\chi\left(\frac{x}{4}\right).$ Then from \eqref{eqn-16}, and \eqref{eqn-17} we have
\begin{align}\label{eqn-101}
A&=-\frac{\varphi(2)}{\Gamma_p(\frac{1}{2})}\sum_{j=0}^{p-2}\omega^j(x)~\pi^{(p-1)\ell_j}~\Gamma_p\left(\left\langle\frac{1}{4}-\frac{j}{p-1}\right\rangle\right)\Gamma_p\left(\left\langle\frac{3}{4}-\frac{j}{p-1}\right\rangle\right)\\
&\times\Gamma_p\left(\left\langle\frac{1}{2}+\frac{j}{p-1}\right\rangle\right)
\Gamma_p\left(\frac{j}{p-1}\right)\notag,
\end{align}
where $\ell_j=1-\left\lfloor\frac{1}{4}-\frac{j}{p-1}\right\rfloor-\left\lfloor\frac{3}{4}-\frac{j}{p-1}\right\rfloor
-\left\lfloor\frac{1}{2}+\frac{j}{p-1}\right\rfloor.$
Now, the term for $j=0$ present in \eqref{eqn-101} is equal to $p\varphi(2)\Gamma_p\left(\frac{1}{4}\right)
\Gamma_p\left(\frac{3}{4}\right).$ Therefore, we have
\begin{align}\label{eqn-102}
A&=-\frac{\varphi(2)}{\Gamma_p(\frac{1}{2})}\sum_{j=1}^{p-2}\omega^j(x)~\pi^{(p-1)\ell_j}~\Gamma_p\left(\left\langle\frac{1}{4}-\frac{j}{p-1}\right\rangle\right)\Gamma_p\left(\left\langle\frac{3}{4}-\frac{j}{p-1}\right\rangle\right)
\\&\times
\Gamma_p\left(\left\langle\frac{1}{2}+\frac{j}{p-1}\right\rangle\right)
\Gamma_p\left(\frac{j}{p-1}\right)+p\varphi(2)\Gamma_p\left(\frac{1}{4}\right)
\Gamma_p\left(\frac{3}{4}\right)\notag.
\end{align}
Using \eqref{lemma-1} we can write
\begin{align}
A&=p\varphi(2)\Gamma_p\left(\frac{1}{4}\right)
\Gamma_p\left(\frac{3}{4}\right)+\frac{\varphi(2)}{\Gamma_p(\frac{1}{2})}\sum_{j=1}^{p-2}\omega^j(-x)~\pi^{(p-1)\ell_j}~\Gamma_p\left(\left\langle\frac{1}{4}-\frac{j}{p-1}\right\rangle\right)\notag
\\&\times\Gamma_p\left(\left\langle\frac{3}{4}-\frac{j}{p-1}\right\rangle\right)
\Gamma_p\left(\left\langle\frac{1}{2}+\frac{j}{p-1}\right\rangle\right)
\Gamma_p\left(\left\langle 1-\frac{j}{p-1}\right\rangle\right)
\Gamma_p\left(\frac{j}{p-1}\right)^2\notag\\
&=-p\varphi(2)\sum_{j=1}^{p-2}\omega^j(-x)
\frac{(-p)^{(-\left\lfloor\frac{1}{2}+\frac{j}{p-1}\right\rfloor)}
\Gamma_p\left(\left\langle\frac{1}{2}+\frac{j}{p-1}\right\rangle\right)}{\Gamma_p(\frac{1}{2})}
\Gamma_p\left(\left\langle 1-\frac{j}{p-1}\right\rangle\right)
\notag\\
&\times(-p)^{-\left\lfloor\frac{1}{4}-\frac{j}{p-1}\right\rfloor-\left\lfloor\frac{3}{4}-\frac{j}{p-1}\right\rfloor}
\Gamma_p\left(\frac{j}{p-1}\right)^2\Gamma_p\left(\left\langle\frac{1}{4}-\frac{j}{p-1}\right\rangle\right)
\Gamma_p\left(\left\langle\frac{3}{4}-\frac{j}{p-1}\right\rangle\right)\notag\\
&+p\varphi(2)\Gamma_p\left(\frac{1}{4}\right)
\Gamma_p\left(\frac{3}{4}\right).\notag
\end{align}
Also, applying \eqref{lemma-2} we obtain
\begin{align}\label{eqn-103}
A&=-\varphi(2)\sum_{t\in\mathbb{F}_p^{\times}}\varphi(t(t-1))\sum_{j=1}^{p-2}\omega^j\left(\frac{x}{t}\right)
(-p)^{-\left\lfloor\frac{1}{4}-\frac{j}{p-1}\right\rfloor-\left\lfloor\frac{3}{4}-\frac{j}{p-1}\right\rfloor}\\
&\times\Gamma_p\left(\left\langle\frac{1}{4}-\frac{j}{p-1}\right\rangle\right)
 \Gamma_p\left(\left\langle\frac{3}{4}-\frac{j}{p-1}\right\rangle\right)\Gamma_p\left(\frac{j}{p-1}\right)^2\notag\\
&~~+p\varphi(2)\Gamma_p\left(\frac{1}{4}\right)
\Gamma_p\left(\frac{3}{4}\right).\notag
\end{align}
The term under summation for $j=0$ is equal to 
\begin{align*}
&-\varphi(2)\displaystyle\sum_{t\in\mathbb{F}_p^{\times}}\varphi(t(t-1))\Gamma_p\left(\frac{1}{4}\right)
\Gamma_p\left(\frac{3}{4}\right)=-\varphi(-2)\Gamma_p\left(\frac{1}{4}\right)\Gamma_p\left(\frac{3}{4}\right)J(\varphi,\varphi)\\&=-p\varphi(2)\Gamma_p\left(\frac{1}{4}\right)\Gamma_p\left(\frac{3}{4}\right){\varphi\choose\varphi}=\varphi(2)
\Gamma_p\left(\frac{1}{4}\right)\Gamma_p\left(\frac{3}{4}\right).
\end{align*}
Note that the last equality is obtained by applying \eqref{rel-1}.
Using this value in \eqref{eqn-103} we obtain
\begin{align}\label{eq-1}
A=(p-1)\varphi(2)\Gamma_p\left(\frac{1}{4}\right)\Gamma_p\left(\frac{3}{4}\right)\left(1+\sum_{t\in\mathbb{F}_p^{\times}}\varphi(t(t-1))
{_2G_2}\left[\begin{array}{cc}
\frac{1}{4}, & \frac{3}{4}\\
0,& 0
\end{array}\mid\frac{t}{x}\right]\right).
\end{align}
Now, from Proposition \ref{prop-1} comparing the values of $A$ we deduce \eqref{formula-1}, \eqref{formula-2}, and 
\eqref{formula-3}. This completes the proof of the theorem.
\begin{proof}[Proof of Theorem \ref{sum-3}]
Applying the transformations \cite[Lemma 3.3]{mccarthy-pacific} and \cite[Proposition 2.5]{mccarthy-ffa}
for $x,t\neq0$ we obtain
\begin{align}\label{eq-2}
{_2G_2}\left[\begin{array}{cc}
\frac{1}{4}, & \frac{3}{4}\\
0,& 0
\end{array}\mid \frac{t}{x}\right]&={\chi_4^3\choose\varepsilon}^{-1}{_2F_1}\left(\begin{array}{cc}
\chi_4, & \chi_4^3\\
& \varepsilon
\end{array}\mid\frac{x}{t}\right)\notag\\&=-p\cdot{_2F_1}\left(\begin{array}{cc}
\chi_4, & \chi_4^3\\
& \varepsilon
\end{array}\mid\frac{x}{t}\right).
\end{align}
Note that we obtain the last equality by using \eqref{rel-1}.
\end{proof}
Let $x=1.$ Then substituting \eqref{eq-2} into \eqref{formula-1} we have
\begin{align}\label{eq-3}
-p\sum_{t\in\mathbb{F}_p^{\times}}\varphi(t(t-1)){_2F_1}\left(\begin{array}{cc}
\chi_4, & \chi_4^3\\
& \varepsilon
\end{array}\mid\frac{1}{t}\right)=-1+\frac{p\varphi(-1)}{\Gamma_p(\frac{1}{4})\Gamma_p(\frac{3}{4})}.
\end{align}
Since $p\equiv1\pmod{4}$ so $\varphi(-1)=1$ and \eqref{prod-3} gives $\Gamma_p(\frac{1}{4})\Gamma_p(\frac{3}{4})=-\chi_4(-1).$ Substituting these two values into \eqref{eq-3} and replacing $t$ by $1/t$ we derive \eqref{formula-4}. Similarly, if $x\neq1$ and $1-x$ is not a square then substituting \eqref{eq-2} into \eqref{formula-2} and replacing $t$ by $x/t$ we obtain \eqref{formula-5}. Finally, if $x\neq1$ and $1-x=a^2$ then substituting \eqref{eq-2} into \eqref{formula-3} and replacing $t$ by $x/t$ we deduce \eqref{formula-6}. This completes the proof.
\end{proof}

\begin{proof}[Proof of Theorem \ref{paf}]
If $1-x\neq\square$ then applying Proposition \ref{prop-1} and Proposition \ref{prop-2} we have
\begin{align*}
{_2G_2}\left[\begin{array}{cc}
\frac{1}{4}, & \frac{3}{4}\vspace{1mm}\\
0, & \frac{1}{2}
\end{array}\mid\frac{1}{x}
\right]={_2G_2}\left[\begin{array}{cc}
\frac{1}{4}, & \frac{3}{4}\vspace{1mm}\\
0, & \frac{1}{2}
\end{array}\mid\frac{x-1}{x}
\right]=0.
\end{align*}
This proves \eqref{paf-1}. Now, let $1-x=a^2$ for some $a\in\mathbb{F}_p^\times.$ Let $a^{-1}$ denote the inverse of $a$ in $\mathbb{F}_p^{\times}.$ Then again applying Proposition \ref{prop-1} and Proposition \ref{prop-2} we have
\begin{align}\label{one}
{_2G_2}\left[\begin{array}{cc}
\frac{1}{4}, & \frac{3}{4}\vspace{1mm}\\
0, & \frac{1}{2}
\end{array}\mid\frac{1}{x}
\right]=-\frac{\varphi(-1)}{\Gamma_p(\frac{1}{4})\Gamma_p(\frac{3}{4})}(\varphi(1+a)+\varphi(1-a)),
\end{align}
and 
\begin{align}\label{two}
{_2G_2}\left[\begin{array}{cc}
\frac{1}{4}, & \frac{3}{4}\vspace{1mm}\\
0, & \frac{1}{2}
\end{array}\mid\frac{x-1}{x}
\right]=-\frac{\varphi(-1)}{\Gamma_p(\frac{1}{4})\Gamma_p(\frac{3}{4})}(\varphi(1+a^{-1})+\varphi(1-a^{-1})).
\end{align}
Comparing \eqref{one} and \eqref{two} we prove \eqref{paf-2}. This completes the proof.
\end{proof}

\section{Concluding remarks}
\begin{remark} 
Let $\mathbb{F}_q$ be a finite field with $q$ elements, where $q=p^r$.
We note that all the transformations and special values of $p$-adic hypergeometric series that are proved in this paper can also be extended to the $q$-version of the $p$-adic hypergeometric series ${_nG_n[\cdots\mid t]_q}$ with $t\in\mathbb{F}_q$ using the definition \cite[Definition 5.1]{mccarthy-pacific}. We avoid this case here for brevity. We also make the same comment for Gaussian hypergeometric functions over $\mathbb{F}_q$. We believe that using this method we can settle many other transformation formulas for $p$-adic hypergeometric series that are analogous to classical hypergeometric series transformations. This is considered as the subject of forthcoming work.
\end{remark}


\end{document}